\documentclass[letterpaper, 10 pt, conference]{ieeeconf}  

\IEEEoverridecommandlockouts                              

\usepackage{graphicx}
\usepackage{mathtools}
\usepackage{amsfonts}
\usepackage{tikz-cd}
\usepackage{amssymb}
\usepackage{import}

\newcommand{\st}{\mathcal{X}} 
\newcommand{\fs}{\mathcal{F}} 
\newcommand{\gd}{\mathcal{G}} 

\newcommand{\Z}{\mathbb{Z}}
\newcommand{\R}{\mathbb{R}}

\newcommand{\Lip}{\textnormal{Lip}}

\newcommand{\dft}{\mathbf{v}}

\newcommand{\ind}{\textnormal{ind}}

\newcommand{\slot}{\,\cdot\,} 

\newcommand{\T}{T}

\newcommand{\interior}{\textnormal{int}}
\newcommand{\dom}{\textnormal{dom}}

\newcommand{\vp}{\varphi}

\newcommand{\Hom}{H}

\newcommand{\CASc}{\v{C}ech-Alexander-Spanier cohomology}

\newcommand{\concept}[1]{\textbf{#1}}

\newcommand{\gi}{\Xi}

\DeclarePairedDelimiter\norm{\lVert}{\rVert}
\DeclareMathOperator{\id}{id}

\newtheorem{Def}{Definition}
\newtheorem{Lem}{Lemma}
\newtheorem{Th}{Theorem}

\newtheorem{Rem}{Remark}

\title{\LARGE \bf
Poincar\'{e}-Hopf Theorem for
Hybrid Systems
}

\author{Matthew D. Kvalheim$^{*}$
\thanks{$^{*}$School of Engineering and Applied Science, University of Pennsylvania,
	Philadelphia, PA 19104, USA}%
}

\begin{document}

\maketitle
\thispagestyle{empty}
\pagestyle{empty}

\begin{abstract}
A generalization of the Poincar\'{e}-Hopf index theorem applicable to hybrid dynamical systems is obtained.
For the hybrid systems considered, guard sets are not assumed to be smooth; distinct ``modes'' are not assumed to have constant dimension; and resets are arbitrary multivalued maps (relations).
\end{abstract}

\section{Introduction}\label{sec:intro}
The Poincar\'{e}-Hopf theorem is a fundamental result of dynamical systems theory and differential topology.
Its simplest form asserts that, for a smooth vector field $\dft$ on a compact boundaryless manifold $M$ with isolated zeros,  the sum of the (Hopf) indices of $\dft$ at its zeros coincides with the Euler characteristic $\chi(M)$ of $M$.
In this paper we obtain a generalization of this theorem which is applicable to a very broad class of hybrid systems.

Hybrid (dynamical) systems are mathematical models of behaviors evolving both continuously and discretely in time.
One motivation for studying hybrid systems comes from robotics and biomechanics, where the making and breaking of contacts intrinsic to most tasks---such as legged locomotion \cite{Holmes_Full_Koditschek_Guckenheimer_2006,revzen2015data,seipel2017conceptual,westervelt2007feedback}---necessitates the introduction of hybrid system models \cite{Koditschek_2021}.
Many investigators have advanced the program\footnote{This program has important contributions from many investigators. We mention only a few here: \cite{Back_Guckenheimer_Myers_1993,Guckenheimer_1995,ye1998stability,alur2000discrete,simic2000towards,lygeros2003dynamical,westervelt2003hybrid, simic2005towards,haghverdi2005bisimulation,ames2005characterization,ames2006categorical, Goebel_Sanfelice_Teel_2009,lerman2016category,Johnson_Burden_Koditschek_2016, pace2017piecewise,lerman2020networks,clark2020existence,clark2021invariant}.} of developing hybrid dynamical systems theory to the same footing as its more mathematically mature parents, the theories of continuous-time or discrete-time dynamical systems.
One approach seeks to generalize results from classical dynamical systems theory to the hybrid setting.\footnote{Examples include extensions of local \cite{simic2001structural} and global \cite{broucke2001structural} structural stability results, contraction analysis \cite{burden2018contraction, burden2018generalizing}, and tools related to periodic orbits \cite{Burden_Sastry_Koditschek_Revzen_2016} such as Floquet theory \cite{Burden_Revzen_Sastry_2015}, averaging theory \cite{De_Burden_Koditschek_2018}, and the Poincar\'{e}-Bendixson theorem \cite{simic2002hybrid, clark2019poincare,clark2020poincare}.}
The present paper represents an additional contribution in this direction.
Note that a Poincar\'{e}-Hopf theorem for hybrid systems different from ours was presented in \cite[Cor.~4]{ames2005homology}. 

From a more abstract point of view, our generalization (Theorem~\ref{th:poinc-hopf-hybrid}) is a relative Poincar\'{e}-Hopf theorem for vector fields on pairs $(\st,\gd)$ satisfying certain properties. 
A well-known extension of the classical Poincar\'{e}-Hopf theorem is to manifolds $M$ with boundary $\partial M$ subject to the constraint that $\dft|_{\partial M}$ is outward-pointing \cite{milnor1965topology}; Pugh removed the outward-pointing restriction for a generic class of vector fields by relating the index sum to a sum of Euler characteristics \cite{pugh1968generalized}.
The latter two results can also be viewed as relative Poincar\'{e}-Hopf theorems for the pair $(M,\partial M)$.
However, the present paper is motivated by the observation that the pairs $(\st,\gd)$ in hybrid systems models arising in applications typically do not enjoy the same smoothness properties; they often have corners or more complicated singularities (cf. \cite[Sec.~2.8]{cgks}).

\section{Main result}

\subsection{Hybrid systems}
In this paper we consider \concept{hybrid systems} $H= (\st, \fs, \gd, \vp, r)$ to be defined by a topological \concept{state space} $\st$, a closed \concept{guard set} $\gd\subset \st$, an open \concept{flow set} $\fs = \st \setminus \gd$, a continuous \emph{local semiflow} $\vp$ on $\fs$ (\S \ref{sec:prelim-standard}), and an arbitrary multivalued \concept{reset} map $r\colon \gd \rightrightarrows \st$ (a set-valued map $r\colon \gd \to 2^{\st}$).
The results of this paper apply to a broad subclass of such systems.
The notation just introduced follows that of \cite{kvalheim2021conley}, but the setup just described is (mostly) more general.
General hybrid systems have both discrete-time and continuous-time behavior; this is captured in typical definitions of \emph{executions} of hybrid systems (Fig.~\ref{fig:execution}).
In the present paper we will not need a formal definition of ``execution''; see \cite[Def.~2.3]{kvalheim2021conley} for one example.

In the above setup there appears to be only a single guard $\gd$, reset $r$, and state space $\st$; however, hybrid systems in the literature are commonly defined using several ``modes'', guards, and resets (and sometimes emphasizing an underlying directed graph structure).
Such systems can be reduced to the above setup by defining $\st$, $\gd$, $\fs$, $\vp$, and $r$ to be the respective unions of all modes, guards, flow sets, local semiflows, and resets. 

\begin{figure}
	\centering
	\def\svgwidth{1.0\columnwidth}
\begingroup%
  \makeatletter%
  \providecommand\color[2][]{%
    \errmessage{(Inkscape) Color is used for the text in Inkscape, but the package 'color.sty' is not loaded}%
    \renewcommand\color[2][]{}%
  }%
  \providecommand\transparent[1]{%
    \errmessage{(Inkscape) Transparency is used (non-zero) for the text in Inkscape, but the package 'transparent.sty' is not loaded}%
    \renewcommand\transparent[1]{}%
  }%
  \providecommand\rotatebox[2]{#2}%
  \newcommand*\fsize{\dimexpr\f@size pt\relax}%
  \newcommand*\lineheight[1]{\fontsize{\fsize}{#1\fsize}\selectfont}%
  \ifx\svgwidth\undefined%
    \setlength{\unitlength}{422.77340183bp}%
    \ifx\svgscale\undefined%
      \relax%
    \else%
      \setlength{\unitlength}{\unitlength * \real{\svgscale}}%
    \fi%
  \else%
    \setlength{\unitlength}{\svgwidth}%
  \fi%
  \global\let\svgwidth\undefined%
  \global\let\svgscale\undefined%
  \makeatother%
  \begin{picture}(1,0.37027798)%
    \lineheight{1}%
    \setlength\tabcolsep{0pt}%
    \put(0,0){\includegraphics[width=\unitlength,page=1]{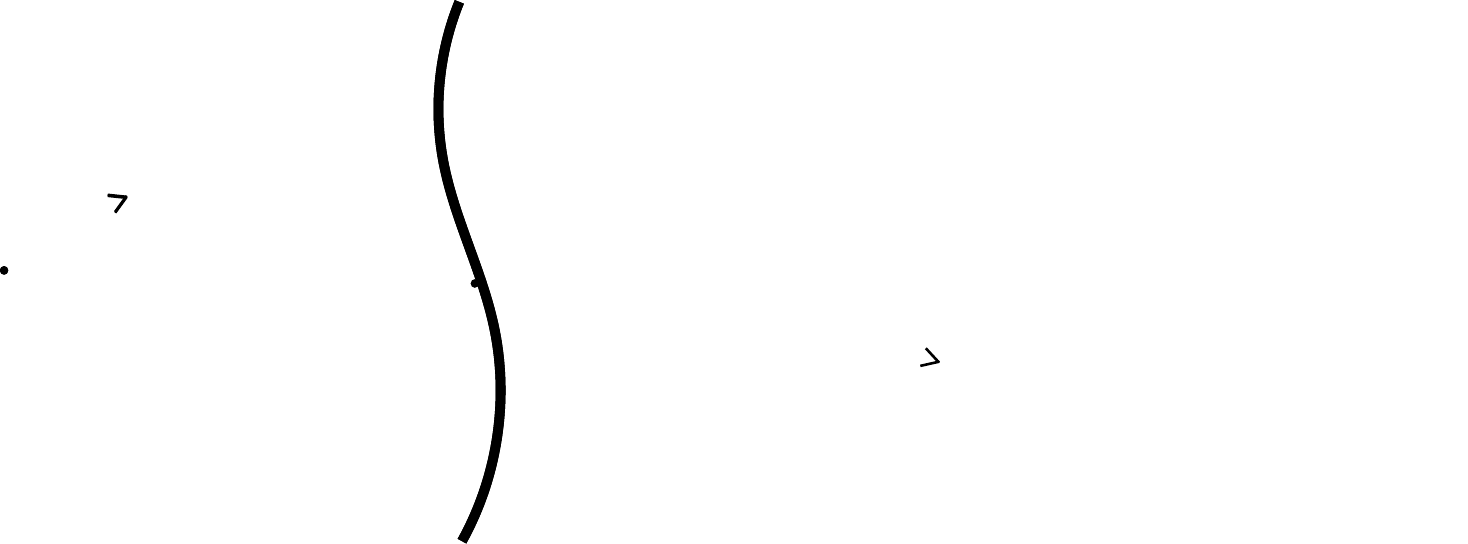}}%
    \put(0.26185865,0.33203108){\color[rgb]{0,0,0}\makebox(0,0)[lt]{\lineheight{1.25}\smash{\begin{tabular}[t]{l}$\gd$\end{tabular}}}}%
    \put(0,0){\includegraphics[width=\unitlength,page=2]{execution-2.pdf}}%
    \put(0.3172859,0.27526249){\color[rgb]{0,0,0}\makebox(0,0)[lt]{\lineheight{1.25}\smash{\begin{tabular}[t]{l}$y\in r(x)$\end{tabular}}}}%
    \put(0,0){\includegraphics[width=\unitlength,page=3]{execution-2.pdf}}%
    \put(0.38469786,0.01271047){\color[rgb]{0,0,0}\makebox(0,0)[lt]{\lineheight{1.25}\smash{\begin{tabular}[t]{l}$z\in r(y)$\end{tabular}}}}%
    \put(0.34921788,0.17946647){\color[rgb]{0,0,0}\makebox(0,0)[lt]{\lineheight{1.25}\smash{\begin{tabular}[t]{l}$x$\end{tabular}}}}%
  \end{picture}%
\endgroup%

	\caption{An execution of a hybrid system $H=(\st,\fs,\gd,\vp,r)$.}\label{fig:execution}
\end{figure}

\subsection{Motivation: stationary point separation principle}
The structure of the invariant sets of a general hybrid system $H = (\st,\fs,\gd,\vp,r)$ depend on $r$ and $\vp$ in a fully coupled (and nonlinear) way.
Hence results involving such invariant sets must take $r$ and $\vp$ into simultaneous consideration in general.
Since stationary points are invariant sets, one might expect that the preceding sentence applies to them.

This is not the case, however, as revealed by the following (self-evident)
\concept{stationary point separation principle (SPSP)} for hybrid systems: the stationary points for $H$ are the disjoint union of the $\vp$-equilibria with the $r$-fixed points; moreover (once $\st$ and $\gd$ are fixed), the $\vp$-equilibria do not depend in any way on $r$, and the $r$-fixed points do not depend in any way on $\vp$.
In other words, the $H$-stationary points separate into two sets which can be analyzed independently.

One might expect the structure of all $H$-stationary points to be strongly constrained by the topology of spaces associated to $H$ such as the hybrifold \cite{simic2000towards,simic2005towards} or hybrid suspension \cite[Fig.~6]{kvalheim2021conley} (or ``homotopy colimit'' \cite{ames2005homology}, or ``$1$-relaxed hybrid quotient space'' \cite{burden2015metrization}).
However, the SPSP implies that the $\vp$-equilibria cannot be constrained by any properties of these spaces depending solely on $r$.

In summary, the reset $r$ is irrelevant to our main (Theorem~\ref{th:poinc-hopf-hybrid}) and subsequent results.
Thus, we formulate these results without any mention of $r$.
\subsection{Statement of results}
The following is our main result; it can be applied to hybrid systems $H = (\st,\fs,\gd,\vp,r)$ satisfying its hypotheses.
In this result and elsewhere in the paper, we follow \cite{tu2010intro,tu2017math-se} by not requiring that the connected components of smooth manifolds with corners have constant dimension.
The standard definitions of local semiflows, the Euler characteristic $\chi(\st)$, the index $\ind_{(-\dft)}(z)$, and (weak) local contractibility are recalled in \S \ref{sec:prelim-standard}; we introduce the \emph{guard index} $\gi_{\dft}(\gd)$ in \S\ref{sec:prelim-guard-index} (Def.~\ref{def:guard-index}).

\begin{Th}\label{th:poinc-hopf-hybrid}
Let $\gd\subset \st$ be compact subspaces of a smooth manifold $M$ with corners and $\vp$ be a continuous local semiflow on $\fs\coloneqq \st\setminus \gd$ satisfying $\frac{\partial}{\partial t} \vp |_{t=0} = \dft|_{\fs}$ for some locally Lipschitz vector field $\dft\colon \st\to \T M$ over $\st$.
Assume that $\st$ is locally contractible and that $(\dft|_\fs)^{-1}(0)$ is a finite set contained in the interior of $\st\setminus \partial M$ in $M\setminus \partial M$.
Then the singular homology $\Hom_{\bullet}(\st)$  is finitely generated and
\begin{equation}\label{eq:th-poinc-hopf-hybrid}
\sum_{z\in \fs \cap  \dft^{-1}(0)}  \ind_{(-\dft)}(z) = \chi(\st)-\gi_{\dft}(\gd).
\end{equation}
\end{Th}

As shown in \cite[Prop.~4.1, 4.3; Cor.~4.4]{kvalheim2021conley}, many hybrid systems appearing in the literature satisfy the \emph{trapping guard} condition introduced in \cite[Def.~3.1]{kvalheim2021conley}.
In the context of Theorem~\ref{th:poinc-hopf-hybrid}, this condition implies the weaker \emph{inflowing} condition on $\gd$ introduced in \S\ref{sec:prelim-guard-index} (Def.~\ref{def:guard-inflowing}), which is easy to check in many examples.
Its satisfaction results in the following simplication of  Theorem~\ref{th:poinc-hopf-hybrid}.

\begin{Th}\label{th:poinc-hopf-inflowing}
Assume the hypotheses of Theorem~\ref{th:poinc-hopf-hybrid}.
Further assume that $\gd$ is locally contractible and inflowing.
Then the singular homologies $\Hom_{\bullet}(\st), \Hom_{\bullet}(\gd)$ are finitely generated and
\begin{equation}\label{eq:th-poinc-hopf-inflowing}
\sum_{z\in \fs \cap  \dft^{-1}(0)}  \ind_{(-\dft)}(z) = \chi(\st)-\chi(\gd).
\end{equation}
\end{Th}
\begin{Rem}
The assumption that $\gd$ is locally contractible can be removed if \v{C}ech-Alexander-Spanier cohomology \cite{spanier1966algebraic,massey1978homology,massey1991basic} is used instead of singular homology to define $\chi(\gd)$.
This follows from the proof of Theorem~\ref{th:poinc-hopf-inflowing}. 
\end{Rem}

\section{Preliminaries}\label{sec:prelim}
In this section we collect preliminaries used in the formulations of Theorem~\ref{th:poinc-hopf-hybrid} and \ref{th:poinc-hopf-inflowing}.

\subsection{Standard notions}\label{sec:prelim-standard}
Let $X$ be a topological space for which the singular homology $\Hom_{\bullet}(X)$ is finitely generated.
We define the \concept{Euler characteristic } $\chi(X)$ of $X$ via (cf. \cite[Thm~2.44]{hatcher2002algebraic})
\begin{equation}\label{eq:euler-cech-def}
\chi(X)\coloneqq \sum_{i} (-1)^i\textnormal{ rank} \Hom_i(X).
\end{equation}

Given a topological space $X$ and $x\in X$, we say that $X$ is \concept{locally contractible at $x$}  if for each neighborhood $U$ of $x$ in $X$ there is a neighborhood $V\subset U$ of $x$ such that the inclusion $V\hookrightarrow U$ is nullhomotopic.
We say that $X$ is \concept{locally contractible} if $X$ is locally contractible at every $x\in X$ (cf. \cite[p.~525]{hatcher2002algebraic}).

Following \cite[Sec.~1.3]{hirsch2006monotone}, a \concept{local semiflow} $\varphi$ on a topological space $\fs$ is a map $\varphi\colon \dom(\varphi) \to \fs$ defined on an open neighborhood $\dom(\varphi) \subseteq [0, \infty) \times \fs$ of $\{0\}\times \fs$ satisfying the following conditions, with $\varphi^t\coloneqq \varphi(t,\slot)$ and $t,s \in [0,\infty)$: (i) $\varphi^0 = \id_\fs$, (ii) $(t+s,x)\in \dom(\varphi) \iff$ both $(s,x)\in \dom(\varphi)$ and  $(t,\varphi^s(x))\in \dom(\varphi)$, and (iii) for all $(t+s,x)\in \dom(\varphi)$, $\varphi^{t+s}(x) = \varphi^t(\varphi^s(x))$.
The local semiflow $\varphi$ is a \concept{semiflow} if $\dom(\varphi) = [0,\infty) \times \fs$.

Let $M$ be a smooth manifold with corners; recall that we do not require the connected components of $M$ to have constant dimension.
Given a continuous vector field $\dft$ on $M$ and a zero $z\in \dft^{-1}(0)\setminus \partial M$ contained in the manifold interior of $M$, the following definition of the (Hopf) \concept{index} follows \cite{milnor1965topology} but (i) works with merely continuous rather than smooth vector fields and (ii) adopts a special definition for the case that $z$ is an isolated point of $M$.
If $\{z\}$ is open in $M$ then the index of $\dft$ at $z$ is defined to be $\ind_{\dft}(z)\coloneqq +1$.
If $\{z\}$ is not open in $M$ then $z$ is contained in a local coordinate chart mapping $z$ to $0\in \R^n$ with $n> 0$, and in local coordinates the function
$$x\mapsto \frac{\dft(x)}{\norm{\dft(x)}}$$
maps a small sphere centered at $z$ into the unit sphere in $\R^n$.
In this case, $\ind_{\dft}(z)$ is defined to be\footnote{If in local coordinates the Jacobian matrix of $\dft$ at $z$ is invertible, then $\ind_{\dft}(z)\in \{-1,+1\}$ is equal to the sign of its determinant.} the degree \cite[p.~134]{hatcher2002algebraic} of this map; it can be shown that $\ind_{\dft}(z)$ is well-defined independently of the choices made.

\subsection{Guard index and inflowing guards}\label{sec:prelim-guard-index}
Throughout this section $\gd\subset \st$ are compact subspaces of a smooth manifold $M$ with corners, and $\vp$ is a continuous local semiflow on $\fs\coloneqq \st\setminus \gd$ satisfying $\frac{\partial}{\partial t} \vp |_{t=0} = \dft|_{\fs}$ for some locally Lipschitz vector field $\dft\colon \st\to \T M$ over $\st$. 
By definition, this means that $\dft$ admits an extension to a locally Lipschitz vector field defined on an open subset of $M$.\footnote{\emph{Local} Lipschitzness of a map between smooth manifolds with corners is a well-defined, metric-independent notion which depends only on the smooth structures of the manifolds involved (cf. \cite[Rem.~1]{kvalheim2021existence}).}
We also assume that $\st$ is locally contractible throughout this section.
We use the notation $\Lip(\st)$ for the set of locally Lipschitz functions $\st\to \R$ on $\st$.
Since $\st$ is compact, every $\alpha\in \Lip(\st)$ is globally Lipschitz.

\begin{Lem}\label{lem:guard-index-indep}
Assume that the closure of $(\dft|_\fs)^{-1}(0)$ is disjoint from $\gd$.
Then for any two nonnegative functions $\alpha,\beta\in \Lip(\st)$ vanishing precisely on $\gd$, there are unique continuous semiflows $\Phi_{\alpha},\Phi_\beta \colon [0,\infty)\times \st \to \st$ satisfying $\frac{\partial}{\partial t}\Phi_{\alpha}|_{t=0}= \alpha\dft, \frac{\partial}{\partial t}\Phi_{\beta}|_{t=0}= \beta\dft$ and a number $\epsilon > 0$ such that, for all $\tau\in (0,\epsilon)$, the fixed point indices \cite{dold1965fixed} of $\Phi_{\alpha}^\tau$ and $\Phi_{\beta}^\tau$ restricted to all open neighborhoods $U\subset \st$ of $\gd$ disjoint from $(\dft|_{\fs})^{-1}(0)$ are equal and independent of the neighborhood.
\end{Lem}

\begin{proof}
Existence and uniqueness of the mentioned semiflows follows from compactness of $\st$, the fact that $\vp$ is a local semiflow, and the Picard-Lindel\"{o}f theorem (after embedding $M$ in some $\R^N$ and extending $\dft$ to a Lipschitz vector field on $\R^N$ using the McShane extension theorem \cite[Thm~1.33]{weaver2018lipschitz}).
Since the closure of $(\dft|_\fs)^{-1}(0)$ is disjoint from $\gd$ it follows that $\gd$ is a connected component of $(\alpha\dft)^{-1}(0) = (\beta\dft)^{-1}(0)$.
The set $(\alpha\dft)^{-1}(0) = (\beta\dft)^{-1}(0)$ is also the fixed point set of both $\Phi_{\alpha}^\tau$ and $\Phi_{\beta}^\tau$ for all sufficiently small $\tau>0$ because the fact that $\alpha\dft$ and $\beta\dft$ are Lipschitz implies that their nonstationary periodic orbits cannot have arbitrarily small periods \cite{yorke1969periods}.  
Thus, there is an open neighborhood $U\subset \st$ of $\gd$ disjoint from $\dft^{-1}(0)\setminus \gd = (\dft|_\fs)^{-1}(0)$, so the fixed point indices of $\Phi_{\alpha}^\tau|_U$ and $\Phi_{\beta}^\tau|_U$ are well-defined.
But $\Phi_{\alpha}^\tau|_U$ and $\Phi_{\beta}^\tau|_U$ are homotopic through maps with constant fixed point set if $\tau$ is small enough since the homotopy $t\mapsto [\alpha + t(\beta-\alpha)] \dft$ from $\alpha \dft$ to $\beta \dft$ has constant zero set (again using \cite{yorke1969periods}).  
By \cite[pp.~2--3]{dold1965fixed} it follows that the fixed point indices of $\Phi_{\alpha}^\tau|_U$ and $\Phi_{\beta}^\tau|_U$ agree, and \cite[Prop.~1.3]{dold1965fixed} implies that these indices are independent of the neighborhood $U\subset \st$ of $\gd$ disjoint from $\dft^{-1}(0)\setminus \gd$.
\end{proof}

\begin{Def}\label{def:guard-index}
Assume that the closure of $(\dft|_\fs)^{-1}(0)$ is disjoint from $\gd$.
Let $\alpha\in \Lip(\st)$ be a nonnegative function vanishing precisely on $\gd$, $\Phi_\alpha\colon [0,\infty)\times \st\to \st$ be the unique continuous semiflow satisfying $\frac{\partial}{\partial t}\Phi_{\alpha}|_{t=0}= \alpha\dft$ of Lem.~\ref{lem:guard-index-indep}, and $U\subset \st$ be an open neighborhood of $\gd$ disjoint from $(\dft|_{\fs})^{-1}(0)$.
Then the \concept{guard index} $\Xi_{\dft}(\gd)\in \Z$ is defined to be the fixed point index \cite{dold1965fixed} of $\Phi_{\alpha}^\tau|_U$ for any sufficiently small $\tau> 0$.
Lem.~\ref{lem:guard-index-indep} implies that $\Xi_{\dft}(\gd)\in \Z$ is independent of the choices made.
\end{Def}

\begin{Def}\label{def:guard-inflowing}
Assume that the closure of $(\dft|_\fs)^{-1}(0)$ is disjoint from $\gd$.
Let $\alpha\in \Lip(\st)$ be a nonnegative function vanishing precisely on $\gd$ and $\Phi_\alpha\colon [0,\infty)\times \st\to \st$ be the unique continuous semiflow satisfying $\frac{\partial}{\partial t}\Phi_{\alpha}|_{t=0}= \alpha\dft$ of Lem.~\ref{lem:guard-index-indep}.
We say that $\gd$ is \concept{inflowing} for $\dft$ if $\gd$ is asymptotically stable for $\Phi_\alpha$. 
This does not depend on the choice of $\alpha$. 
\end{Def}

\section{Proofs of Theorems~\ref{th:poinc-hopf-hybrid} and \ref{th:poinc-hopf-inflowing}}

\subsection{Proof of Theorem~\ref{th:poinc-hopf-hybrid}}
Let $\alpha\in \Lip(\st)$, $\Phi_\alpha$ and $U$ be as in Def.~\ref{def:guard-index}.
First, local contractibility and compactness of $\st$ along with embeddability of $M\supset \st$ in some $\R^N$ imply that $\st$ is a compact Euclidean neighborhood retract (ENR) \cite[Thm~A.7]{hatcher2002algebraic}, so its singular homology is finitely generated \cite[Cor.~A.8]{hatcher2002algebraic}.
Next, each time-$\tau$ map $\Phi_\alpha^\tau\colon \st\to \st$ is homotopic to the identity.
Moreover, since $\alpha \dft$ is Lipschitz it follows that the fixed point set of $\Phi_{\alpha}^\tau$ coincides with $(\alpha \dft)^{-1}(0)$ if $\tau >0$ is sufficiently small \cite{yorke1969periods}.

Thus, the Lefschetz fixed point theorem  \cite[Thm~4.1]{dold1965fixed} and the definition of $\gi_{\dft}(\gd)$ imply the following equality for all sufficiently small $\tau > 0$:
\begin{equation}\label{eq:th-general-1}
\gi_{\dft}(\gd) + I_{\Phi_{\alpha}^\tau|_{\fs}} = \chi(\st),
\end{equation}
where $I_{\Phi_{\alpha}^\tau|_{\fs}}$ is the fixed point index \cite{dold1965fixed} of $\Phi_{\alpha}^\tau|_{\fs}\colon \fs\to \fs$.

Fix $z\in (\dft|_\fs)^{-1}(0)$ and let $V\subset \fs$ be the open domain of a coordinate chart for $M\setminus \partial M$ disjoint from the other zeros sending $z$ to $0\in \R^{n_z}$.
Then for all sufficiently small $\tau > 0$, in local coordinates for $V$ the function
$$x\mapsto F^{\tau}(x)\coloneqq \frac{x-\Phi_{\alpha}^\tau(x)}{\norm{x-\Phi_{\alpha}^\tau(x)}}$$
maps a small sphere centered at $z$ into the unit sphere in $\R^{n_z}$, and the degree of this map is equal to the fixed point index of $\Phi_{\alpha}^\tau|_V$.
Moreover,
$$H(t,x)\coloneqq \begin{cases}
\frac{-\dft(x)}{\norm{-\dft(x)}}, & t=0\\
\frac{1}{2t + (1-2t)\alpha(x)} F^{t\tau}(x), & 0 < t \leq 1/2\\
F^{t\tau}(x), & 1/2\leq t \leq 1
\end{cases}$$
is a continuous homotopy from $-\dft/\norm{\dft}$ to $F^\tau$, so the fixed point index of $\Phi_{\alpha}^\tau|_V$ is equal to $\ind_{-\dft}(z)$.
This, \eqref{eq:th-general-1}, and additivity of the fixed point index \cite[pp.~1,~3]{dold1965fixed} imply the desired equality \eqref{eq:th-poinc-hopf-hybrid}.

\subsection{Proof of Theorem~\ref{th:poinc-hopf-inflowing}}

From Theorem~\ref{th:poinc-hopf-hybrid}  it suffices to prove that $\gi_{\dft}(\gd)=\chi(\gd)$ under the additional assumptions that $\gd$ is and locally contractible and inflowing.
Fix any $\alpha$ and corresponding $\Phi_{\alpha}$ as in Def.~\ref{def:guard-inflowing}.
The set $\gd$ is asymptotically stable for $\Phi_{\alpha}$ since $\gd$ is inflowing, so there is a neighborhood $V\subset \st$ of $\gd$ which is a compact Euclidean neighborhood retract (ENR) satisfying $\Phi_{\alpha}^t(V)\subset \interior(V)$ for all $t > 0$ \cite[Prop.~2.4]{wang2016shape} and containing no zeros of $(\alpha\dft)^{-1}(0)\setminus \gd$.\footnote{In more detail, take $V$ to be a compact sublevel set of a continuous Lyapunov function for $\gd$ with respect to $\Phi_\alpha|_{[0,\infty) \times B(\gd)}$ \cite{wang2016shape}, where $B(\gd)$ is the basin of attraction of $\gd$ for $\Phi_\alpha$. 
The set $V$ is a strong deformation retract of $B(\gd)$ since $B(\gd)\setminus \interior(V)$ is a Wazewski set for $\Phi_\alpha|_{[0,\infty) \times B(\gd)}$ \cite[Thm~II.2.3]{conley1978isolated}. 
(Wazewski's theorem is stated in \cite{conley1978isolated} for flows, but exactly the same proof works for semiflows).
In particular, the compact $V$ is a retract of the open ENR $B(\gd)\subset \st$, so $V$ is a compact ENR.}
Thus, the Lefschetz fixed point theorem \cite[Thm~4.1]{dold1965fixed} and \cite{yorke1969periods} imply that $\chi(V)$ is equal to the fixed point index of $\Phi_{\alpha}^\tau|_{V}$ for all sufficiently small $\tau > 0$, and for sufficiently small $\tau > 0$ this fixed point index is equal to $\gi_{\dft}(\gd)$ (by Def.~\ref{def:guard-index}).
Since $V$ is an ENR its real singular homology is isomorphic to its real \CASc~\cite[p.~285, Prop.~6.12]{dold1995lectures}, and the same is true of the ENR $\gd$.
Since the real \v{C}ech-Alexander-Spanier cohomologies of $V$ and $\gd$ are isomorphic \cite[Thm~6.3]{gobbino2001topological}, the desired equality $\chi(\gd) = \chi(V)=\gi_{\dft}(\gd)$ follows.

\section{Conclusion}\label{sec:conclusion}
We have obtained a generalization (Theorem~\ref{th:poinc-hopf-hybrid}) of the classical Poincar\'{e}-Hopf theorem which, along with its inflowing specialization (Theorem~\ref{th:poinc-hopf-inflowing}), are applicable to a very broad class of hybrid systems. 
While Theorem~\ref{th:poinc-hopf-hybrid} generalizes the classical Poincar\'{e}-Hopf theorem for smooth vector fields, it is not a complete generalization of the Poincar\'{e}-Hopf theorem for continuous vector fields because we assume that the vector field is locally Lipschitz.
The reason for this assumption is that we proved Theorem~\ref{th:poinc-hopf-hybrid} by applying the Lefschetz fixed point theorem to the time-$\tau$ maps of a certain semiflow on $\st$.
Using a result of Yorke \cite{yorke1969periods}, the Lipschitz assumption is used to rule out the possibility that there are nonstationary periodic orbits of arbitrarily small periods, and this enables us to infer that the fixed point set of the time-$\tau$ maps coincide with $\dft^{-1}(0)$ for sufficiently small $\tau$.

A different generalization of the Poincar\'{e}-Hopf theorem not mentioned in \S \ref{sec:intro} is due to McCord \cite{mccord1989hopf}; this result says that the sum of the Hopf indices of those zeros of a smooth vector field contained in an isolated invariant set $S$ is equal to the Euler characteristic of the homology Conley index of $S$.
It would be interesting to develop a Conley index theory \cite{conley1978isolated,mischaikow1999conley,mischaikow2002conley} for hybrid systems and to investigate whether an analogous relationship holds in the hybrid setting.
A Conley index theory for hybrid systems would also be interesting since, together with the results of \cite{kvalheim2021conley}, it would constitute a first theoretical step toward extending rigorous computer-assisted analysis methods based on Conley theory \cite{mischaikow2002topological,kalies2005algorithmic} to hybrid dynamical systems.

The Poincar\'{e}-Hopf theorem relates the equilibria of flows generated by vector fields to the topology of the underlying state space.
Similarly, the Conley index furnishes global topological constraints on the isolated invariant sets of a flow via the generalized Morse-Smale theorem \cite[Sec.~IV.5.3]{conley1978isolated}, but these constraints are relatively coarse.
The constraints on equilibria obtained via the Poincar\'{e}-Hopf theorem are sharper; sharper topological constraints on nonstationary periodic orbits have also been obtained using tools such as the Fuller index \cite{fuller1967index}.
If such tools could be extended to hybrid systems, they might be useful for proving existence of periodic orbits such as those representing steady gaits in legged locomotion \cite{Burden_Sastry_Koditschek_Revzen_2016, morris2005restricted,hamed2016exponentially}.
Another approach to proving existence of periodic orbits is via the global bifurcation or continuation theory initiated by Alexander, Yorke, and others \cite{alexander1978global,mallet1982snakes,alligood1984families,fiedler1988global,kvalheim2021families}; extending this theory to hybrid systems may also prove fruitful.

In closing, we mention that the present paper was partially motivated by the generalization of Brockett's necessary condition for stabilizability of points \cite[Thm~1.(iii)]{brockett1983asymptotic} to one for stabilization of more general sets for continuous-time control systems in \cite[Thm~1]{kvalheim2021necessary}.
Using this generalization, necessary conditions for a continuous-time system to operate safely relative to some subset of state space were also derived \cite[Thm~2]{kvalheim2021necessary}.
A key tool used in the proofs of these generalizations is the classical Poincar\'{e}-Hopf theorem; thus, it seems that Theorems~\ref{th:poinc-hopf-hybrid} and \ref{th:poinc-hopf-inflowing} of the present paper might enable these generalizations to be extended to hybrid control systems.

\section*{Acknowledgments}
This work is supported by the Army Research Office (ARO) under the SLICE Multidisciplinary University Research Initiatives (MURI) Program, award W911NF1810327.
The author gratefully acknowledges helpful conversations with Paul Gustafson and Daniel E. Koditschek.

\bibliographystyle{IEEEtran}
\bibliography{ref}

\begin{thebibliography}{10}
\providecommand{\url}[1]{#1}
\csname url@samestyle\endcsname
\providecommand{\newblock}{\relax}
\providecommand{\bibinfo}[2]{#2}
\providecommand{\BIBentrySTDinterwordspacing}{\spaceskip=0pt\relax}
\providecommand{\BIBentryALTinterwordstretchfactor}{4}
\providecommand{\BIBentryALTinterwordspacing}{\spaceskip=\fontdimen2\font plus
\BIBentryALTinterwordstretchfactor\fontdimen3\font minus
  \fontdimen4\font\relax}
\providecommand{\BIBforeignlanguage}[2]{{%
\expandafter\ifx\csname l@#1\endcsname\relax
\typeout{** WARNING: IEEEtran.bst: No hyphenation pattern has been}%
\typeout{** loaded for the language `#1'. Using the pattern for}%
\typeout{** the default language instead.}%
\else
\language=\csname l@#1\endcsname
\fi
#2}}
\providecommand{\BIBdecl}{\relax}
\BIBdecl

\bibitem{Holmes_Full_Koditschek_Guckenheimer_2006}
P.~Holmes, R.~J. Full, D.~E. Koditschek, and J.~Guckenheimer, ``The dynamics of
  legged locomotion: Models, analyses, and challenges,'' \emph{SIAM Review},
  vol.~48, no.~2, p. 207–304, 2006.

\bibitem{revzen2015data}
S.~Revzen and M.~Kvalheim, ``Data driven models of legged locomotion,'' in
  \emph{Micro-and Nanotechnology Sensors, Systems, and Applications VII}, vol.
  9467.\hskip 1em plus 0.5em minus 0.4em\relax International Society for Optics
  and Photonics, 2015, p. 94671V.

\bibitem{seipel2017conceptual}
J.~Seipel, M.~Kvalheim, S.~Revzen, M.~A. Sharbafi, and A.~Seyfarth,
  ``Conceptual models of legged locomotion,'' in \emph{Bioinspired Legged
  Locomotion}.\hskip 1em plus 0.5em minus 0.4em\relax Elsevier, 2017, pp.
  55--131.

\bibitem{westervelt2007feedback}
E.~R. Westervelt, J.~W. Grizzle, C.~Chevallereau, J.~H. Choi, and B.~Morris,
  \emph{Feedback control of dynamic bipedal robot locomotion}.\hskip 1em plus
  0.5em minus 0.4em\relax Taylor \& Francis/CRC, 2007.

\bibitem{Koditschek_2021}
\BIBentryALTinterwordspacing
D.~E. Koditschek, ``What is robotics? {W}hy do we need it, and how can we get
  it?'' \emph{Annual Review of Control, Robotics, and Autonomous Systems},
  vol.~4, no. (to appear), p. 1–37, 2021. [Online]. Available:
  \url{https://doi.org/10.1146/annurev-control- 080320-011601}
\BIBentrySTDinterwordspacing

\bibitem{Back_Guckenheimer_Myers_1993}
A.~Back, J.~Guckenheimer, and M.~Myers, ``A dynamical simulation facility for
  hybrid systems,'' \emph{Hybrid Systems}, p. 255–267, 1993.

\bibitem{Guckenheimer_1995}
J.~Guckenheimer, ``A robust hybrid stabilization strategy for equilibria,''
  \emph{IEEE Transactions on Automatic Control}, vol.~40, no.~2, p. 321–326,
  1995.

\bibitem{ye1998stability}
H.~Ye, A.~N. Michel, and L.~Hou, ``Stability theory for hybrid dynamical
  systems,'' \emph{IEEE transactions on automatic control}, vol.~43, no.~4, pp.
  461--474, 1998.

\bibitem{alur2000discrete}
R.~Alur, T.~A. Henzinger, G.~Lafferriere, and G.~J. Pappas, ``Discrete
  abstractions of hybrid systems,'' \emph{Proceedings of the IEEE}, vol.~88,
  no.~7, pp. 971--984, 2000.

\bibitem{simic2000towards}
S.~N. Simi{\'c}, K.~H. Johansson, S.~Sastry, and J.~Lygeros, ``Towards a
  geometric theory of hybrid systems,'' in \emph{International Workshop on
  Hybrid Systems: Computation and Control}.\hskip 1em plus 0.5em minus
  0.4em\relax Springer, 2000, pp. 421--436.

\bibitem{lygeros2003dynamical}
J.~Lygeros, K.~H. Johansson, S.~N. Simic, J.~Zhang, and S.~S. Sastry,
  ``Dynamical properties of hybrid automata,'' \emph{IEEE Transactions on
  automatic control}, vol.~48, no.~1, pp. 2--17, 2003.

\bibitem{westervelt2003hybrid}
E.~R. Westervelt, J.~W. Grizzle, and D.~E. Koditschek, ``Hybrid zero dynamics
  of planar biped walkers,'' \emph{IEEE transactions on automatic control},
  vol.~48, no.~1, pp. 42--56, 2003.

\bibitem{simic2005towards}
S.~N. Simi{\'c}, K.~H. Johansson, J.~Lygeros, and S.~Sastry, ``Towards a
  geometric theory of hybrid systems,'' \emph{Dynamics of Continuous, Discrete
  and Impulsive Systems Series B: Applications and Algorithms}, vol.~12, no.
  5-6, pp. 649--687, 2005.

\bibitem{haghverdi2005bisimulation}
E.~Haghverdi, P.~Tabuada, and G.~J. Pappas, ``Bisimulation relations for
  dynamical, control, and hybrid systems,'' \emph{Theoretical Computer
  Science}, vol. 342, no. 2-3, pp. 229--261, 2005.

\bibitem{ames2005characterization}
A.~D. Ames and S.~Sastry, ``Characterization of zeno behavior in hybrid systems
  using homological methods,'' in \emph{Proceedings of the 2005, American
  Control Conference, 2005.}\hskip 1em plus 0.5em minus 0.4em\relax IEEE, 2005,
  pp. 1160--1165.

\bibitem{ames2006categorical}
A.~D. Ames, \emph{A categorical theory of hybrid systems}.\hskip 1em plus 0.5em
  minus 0.4em\relax ProQuest LLC, Ann Arbor, MI, 2006, thesis
  (Ph.D.)--University of California, Berkeley.

\bibitem{Goebel_Sanfelice_Teel_2009}
R.~Goebel, R.~G. Sanfelice, and A.~Teel, ``Hybrid dynamical systems,''
  \emph{Control Systems, IEEE}, vol.~29, no.~2, p. 28–93, 2009.

\bibitem{lerman2016category}
E.~Lerman, ``A category of hybrid systems,'' \emph{arXiv preprint
  arXiv:1612.01950}, 2016.

\bibitem{Johnson_Burden_Koditschek_2016}
A.~M. Johnson, S.~A. Burden, and D.~E. Koditschek, ``A hybrid systems model for
  simple manipulation and self-manipulation systems,'' \emph{The International
  Journal of Robotics Research}, vol.~35, no.~11, pp. 1354--1392, Sep 2016.

\bibitem{pace2017piecewise}
A.~M. Pace and S.~A. Burden, ``Piecewise-differentiable trajectory outcomes in
  mechanical systems subject to unilateral constraints,'' in \emph{Proceedings
  of the 20th International Conference on Hybrid Systems: Computation and
  Control}, 2017, pp. 243--252.

\bibitem{lerman2020networks}
\BIBentryALTinterwordspacing
E.~Lerman and J.~Schmidt, ``Networks of hybrid open systems,'' \emph{J. Geom.
  Phys.}, vol. 149, pp. 103\,582, 30, 2020. [Online]. Available:
  \url{https://doi.org/10.1016/j.geomphys.2019.103582}
\BIBentrySTDinterwordspacing

\bibitem{clark2020existence}
W.~Clark and A.~Bloch, ``Existence of invariant volumes in nonholonomic
  systems,'' \emph{arXiv preprint arXiv:2009.11387}, 2020.

\bibitem{clark2021invariant}
------, ``Invariant forms in hybrid and impact systems,'' \emph{arXiv preprint
  arXiv:2101.11128}, 2021.

\bibitem{simic2001structural}
S.~N. Simi{\'c}, K.~H. Johansson, J.~Lygeros, and S.~Sastry, ``Structural
  stability of hybrid systems,'' in \emph{2001 European Control Conference
  (ECC)}.\hskip 1em plus 0.5em minus 0.4em\relax IEEE, 2001, pp. 3858--3863.

\bibitem{broucke2001structural}
M.~E. Broucke, C.~C. Pugh, and S.~N. Simi{\'c}, ``Structural stability of
  piecewise smooth systems,'' \emph{Computational and applied mathematics},
  vol.~20, no. 1-2, pp. 51--89, 2001.

\bibitem{burden2018contraction}
S.~A. Burden, T.~Libby, and S.~D. Coogan, ``On contraction analysis for hybrid
  systems,'' \emph{arXiv preprint arXiv:1811.03956}, 2018.

\bibitem{burden2018generalizing}
S.~A. Burden and S.~D. Coogan, ``Generalizing infinitesimal contraction
  analysis to hybrid systems,'' \emph{arXiv preprint arXiv:1804.04122}, 2018.

\bibitem{Burden_Sastry_Koditschek_Revzen_2016}
\BIBentryALTinterwordspacing
S.~A. Burden, S.~S. Sastry, D.~E. Koditschek, and S.~Revzen, ``Event-selected
  vector field discontinuities yield piecewise-differentiable flows,''
  \emph{SIAM J. Appl. Dyn. Syst.}, vol.~15, no.~2, pp. 1227--1267, 2016.
  [Online]. Available: \url{https://doi.org/10.1137/15M1016588}
\BIBentrySTDinterwordspacing

\bibitem{Burden_Revzen_Sastry_2015}
\BIBentryALTinterwordspacing
S.~A. Burden, S.~Revzen, and S.~S. Sastry, ``Model reduction near periodic
  orbits of hybrid dynamical systems,'' \emph{IEEE Trans. Automat. Control},
  vol.~60, no.~10, pp. 2626--2639, 2015. [Online]. Available:
  \url{https://doi.org/10.1109/TAC.2015.2411971}
\BIBentrySTDinterwordspacing

\bibitem{De_Burden_Koditschek_2018}
A.~De, S.~A. Burden, and D.~E. Koditschek, ``A hybrid dynamical extension of
  averaging and its application to the analysis of legged gait stability,''
  \emph{The International Journal of Robotics Research}, vol.~37, no. 2–3, p.
  266–286, Mar 2018.

\bibitem{simic2002hybrid}
S.~N. Simi{\'c}, S.~Sastry, K.~H. Johansson, and J.~Lygeros, ``Hybrid limit
  cycles and hybrid {P}oincar{\'e}-{B}endixson,'' \emph{IFAC Proceedings
  Volumes}, vol.~35, no.~1, pp. 197--202, 2002.

\bibitem{clark2019poincare}
W.~Clark, A.~Bloch, and L.~Colombo, ``A {P}oincar{\'e}-{B}endixson theorem for
  hybrid systems,'' \emph{Mathematical Control \& Related Fields}, vol.~10,
  no.~1, pp. 27--45, 2019.

\bibitem{clark2020poincare}
W.~Clark and A.~M. Bloch, ``A {P}oincar{\'e}--{B}endixson theorem for hybrid
  dynamical systems on directed graphs,'' \emph{Mathematics of Control,
  Signals, and Systems}, vol.~32, no.~1, p.~1, 2020.

\bibitem{ames2005homology}
A.~D. Ames and S.~Sastry, ``A homology theory for hybrid systems: Hybrid
  homology,'' in \emph{International Workshop on Hybrid Systems: Computation
  and Control}.\hskip 1em plus 0.5em minus 0.4em\relax Springer, 2005, pp.
  86--102.

\bibitem{milnor1965topology}
J.~W. Milnor, \emph{Topology from the differentiable viewpoint}, ser. Princeton
  Landmarks in Mathematics.\hskip 1em plus 0.5em minus 0.4em\relax Princeton
  University Press, Princeton, NJ, 1997, based on notes by David W. Weaver,
  Revised reprint of the 1965 original.

\bibitem{pugh1968generalized}
\BIBentryALTinterwordspacing
C.~C. Pugh, ``A generalized {P}oincar\'{e} index formula,'' \emph{Topology},
  vol.~7, pp. 217--226, 1968. [Online]. Available:
  \url{https://doi.org/10.1016/0040-9383(68)90002-5}
\BIBentrySTDinterwordspacing

\bibitem{cgks}
J.~Culbertson, P.~Gustafson, D.~E. Koditschek, and P.~F. Stiller, ``Formal
  composition of hybrid systems,'' \emph{Theory Appl. Categ.}, vol.~35, pp.
  Paper No. 45, 1634--1682, 2020.

\bibitem{kvalheim2021conley}
\BIBentryALTinterwordspacing
M.~D. Kvalheim, P.~Gustafson, and D.~E. Koditschek, ``Conley's fundamental
  theorem for a class of hybrid systems,'' \emph{SIAM J. Appl. Dyn. Syst.},
  vol.~20, no.~2, pp. 784--825, 2021. [Online]. Available:
  \url{https://doi.org/10.1137/20M1336576}
\BIBentrySTDinterwordspacing

\bibitem{burden2015metrization}
S.~A. Burden, H.~Gonzalez, R.~Vasudevan, R.~Bajcsy, and S.~S. Sastry,
  ``Metrization and simulation of controlled hybrid systems,'' \emph{IEEE
  Transactions on Automatic Control}, vol.~60, no.~9, pp. 2307--2320, 2015.

\bibitem{tu2010intro}
L.~W. Tu, \emph{Introduction to Manifolds}, 2nd~ed.\hskip 1em plus 0.5em minus
  0.4em\relax Springer Science \& Business Media, 2010.

\bibitem{tu2017math-se}
\BIBentryALTinterwordspacing
L.~Tu, ``Can a topological manifold be non-connected and each component with
  different dimension?'' Mathematics Stack Exchange, 2017,
  uRL:https://math.stackexchange.com/q/3208774 (version: 2019-04-30). [Online].
  Available: \url{https://math.stackexchange.com/q/3208774}
\BIBentrySTDinterwordspacing

\bibitem{spanier1966algebraic}
E.~H. Spanier, \emph{Algebraic topology}.\hskip 1em plus 0.5em minus
  0.4em\relax Springer-Verlag, New York, [1995?], corrected reprint of the 1966
  original.

\bibitem{massey1978homology}
W.~S. Massey, \emph{Homology and Cohomology Theory: an approach based on
  {A}lexander-{S}panier cochains}.\hskip 1em plus 0.5em minus 0.4em\relax
  Marcel Dekker, 1978, vol.~46.

\bibitem{massey1991basic}
------, \emph{A basic course in algebraic topology}.\hskip 1em plus 0.5em minus
  0.4em\relax Springer-Verlag, 1991.

\bibitem{hatcher2002algebraic}
A.~Hatcher, \emph{Algebraic topology}.\hskip 1em plus 0.5em minus 0.4em\relax
  Cambridge University Press, Cambridge, 2002.

\bibitem{hirsch2006monotone}
M.~W. Hirsch and H.~Smith, ``Monotone dynamical systems,'' in \emph{Handbook of
  differential equations: ordinary differential equations. {V}ol. {II}}.\hskip
  1em plus 0.5em minus 0.4em\relax Elsevier B. V., Amsterdam, 2005, pp.
  239--357.

\bibitem{kvalheim2021existence}
\BIBentryALTinterwordspacing
M.~D. Kvalheim and S.~Revzen, ``Existence and uniqueness of global {K}oopman
  eigenfunctions for stable fixed points and periodic orbits,'' \emph{Phys. D},
  vol. 425, pp. Paper No. 132\,959, 20, 2021. [Online]. Available:
  \url{https://doi.org/10.1016/j.physd.2021.132959}
\BIBentrySTDinterwordspacing

\bibitem{dold1965fixed}
\BIBentryALTinterwordspacing
A.~Dold, ``Fixed point index and fixed point theorem for {E}uclidean
  neighborhood retracts,'' \emph{Topology}, vol.~4, pp. 1--8, 1965. [Online].
  Available: \url{https://doi.org/10.1016/0040-9383(65)90044-3}
\BIBentrySTDinterwordspacing

\bibitem{weaver2018lipschitz}
N.~Weaver, \emph{Lipschitz algebras}.\hskip 1em plus 0.5em minus 0.4em\relax
  World Scientific Publishing Co. Pte. Ltd., Hackensack, NJ, 2018, second
  edition of [ MR1832645].

\bibitem{yorke1969periods}
\BIBentryALTinterwordspacing
J.~A. Yorke, ``Periods of periodic solutions and the {L}ipschitz constant,''
  \emph{Proc. Amer. Math. Soc.}, vol.~22, pp. 509--512, 1969. [Online].
  Available: \url{https://doi.org/10.2307/2037090}
\BIBentrySTDinterwordspacing

\bibitem{wang2016shape}
\BIBentryALTinterwordspacing
J.~Wang, D.~Li, and J.~Duan, ``On the shape {C}onley index theory of semiflows
  on complete metric spaces,'' \emph{Discrete Contin. Dyn. Syst.}, vol.~36,
  no.~3, pp. 1629--1647, 2016. [Online]. Available:
  \url{https://doi.org/10.3934/dcds.2016.36.1629}
\BIBentrySTDinterwordspacing

\bibitem{dold1995lectures}
\BIBentryALTinterwordspacing
A.~Dold, \emph{Lectures on algebraic topology}, ser. Classics in
  Mathematics.\hskip 1em plus 0.5em minus 0.4em\relax Springer-Verlag, Berlin,
  1995, reprint of the 1972 edition. [Online]. Available:
  \url{https://doi.org/10.1007/978-3-642-67821-9}
\BIBentrySTDinterwordspacing

\bibitem{gobbino2001topological}
\BIBentryALTinterwordspacing
M.~Gobbino, ``Topological properties of attractors for dynamical systems,''
  \emph{Topology}, vol.~40, no.~2, pp. 279--298, 2001. [Online]. Available:
  \url{https://doi.org/10.1016/S0040-9383(99)00061-0}
\BIBentrySTDinterwordspacing

\bibitem{mccord1989hopf}
\BIBentryALTinterwordspacing
C.~K. McCord, ``On the {H}opf index and the {C}onley index,'' \emph{Trans.
  Amer. Math. Soc.}, vol. 313, no.~2, pp. 853--860, 1989. [Online]. Available:
  \url{https://doi.org/10.2307/2001432}
\BIBentrySTDinterwordspacing

\bibitem{conley1978isolated}
C.~Conley, \emph{Isolated invariant sets and the {M}orse index}, ser. CBMS
  Regional Conference Series in Mathematics.\hskip 1em plus 0.5em minus
  0.4em\relax American Mathematical Society, Providence, R.I., 1978, vol.~38.

\bibitem{mischaikow1999conley}
K.~Mischaikow, ``The {C}onley index theory: a brief introduction,'' in
  \emph{Conley index theory ({W}arsaw, 1997)}, ser. Banach Center Publ.\hskip
  1em plus 0.5em minus 0.4em\relax Polish Acad. Sci. Inst. Math., Warsaw, 1999,
  vol.~47, pp. 9--19.

\bibitem{mischaikow2002conley}
\BIBentryALTinterwordspacing
K.~Mischaikow and M.~Mrozek, ``Conley index,'' in \emph{Handbook of dynamical
  systems, {V}ol. 2}.\hskip 1em plus 0.5em minus 0.4em\relax North-Holland,
  Amsterdam, 2002, pp. 393--460. [Online]. Available:
  \url{https://doi.org/10.1016/S1874-575X(02)80030-3}
\BIBentrySTDinterwordspacing

\bibitem{mischaikow2002topological}
\BIBentryALTinterwordspacing
K.~Mischaikow, ``Topological techniques for efficient rigorous computation in
  dynamics,'' \emph{Acta Numer.}, vol.~11, pp. 435--477, 2002. [Online].
  Available: \url{https://doi.org/10.1017/S0962492902000065}
\BIBentrySTDinterwordspacing

\bibitem{kalies2005algorithmic}
\BIBentryALTinterwordspacing
W.~D. Kalies, K.~Mischaikow, and R.~C. A.~M. VanderVorst, ``An algorithmic
  approach to chain recurrence,'' \emph{Found. Comput. Math.}, vol.~5, no.~4,
  pp. 409--449, 2005. [Online]. Available:
  \url{https://doi.org/10.1007/s10208-004-0163-9}
\BIBentrySTDinterwordspacing

\bibitem{fuller1967index}
F.~B. Fuller, ``An index of fixed point type for periodic orbits,''
  \emph{American Journal of Mathematics}, vol.~89, no.~1, pp. 133--148, 1967.

\bibitem{morris2005restricted}
B.~Morris and J.~W. Grizzle, ``A restricted poincar{\'e} map for determining
  exponentially stable periodic orbits in systems with impulse effects:
  Application to bipedal robots,'' in \emph{Proceedings of the 44th IEEE
  Conference on Decision and Control}.\hskip 1em plus 0.5em minus 0.4em\relax
  IEEE, 2005, pp. 4199--4206.

\bibitem{hamed2016exponentially}
K.~A. Hamed, B.~G. Buss, and J.~W. Grizzle, ``Exponentially stabilizing
  continuous-time controllers for periodic orbits of hybrid systems:
  Application to bipedal locomotion with ground height variations,'' \emph{The
  International Journal of Robotics Research}, vol.~35, no.~8, pp. 977--999,
  2016.

\bibitem{alexander1978global}
J.~C. Alexander and J.~A. Yorke, ``Global bifurcations of periodic orbits,''
  \emph{American Journal of Mathematics}, vol. 100, no.~2, pp. 263--292, 1978.

\bibitem{mallet1982snakes}
J.~Mallet-Paret and J.~A. Yorke, ``Snakes: oriented families of periodic
  orbits, their sources, sinks, and continuation,'' \emph{Journal of
  Differential Equations}, vol.~43, no.~3, pp. 419--450, 1982.

\bibitem{alligood1984families}
K.~T. Alligood and J.~A. Yorke, ``Families of periodic orbits: virtual periods
  and global continuability,'' \emph{Journal of differential equations},
  vol.~55, no.~1, pp. 59--71, 1984.

\bibitem{fiedler1988global}
\BIBentryALTinterwordspacing
B.~Fiedler, \emph{Global bifurcation of periodic solutions with symmetry}, ser.
  Lecture Notes in Mathematics.\hskip 1em plus 0.5em minus 0.4em\relax
  Springer-Verlag, Berlin, 1988, vol. 1309. [Online]. Available:
  \url{https://doi.org/10.1007/BFb0082943}
\BIBentrySTDinterwordspacing

\bibitem{kvalheim2021families}
\BIBentryALTinterwordspacing
M.~D. Kvalheim and A.~M. Bloch, ``Families of periodic orbits: closed 1-forms
  and global continuability,'' \emph{J. Differential Equations}, vol. 285, pp.
  211--257, 2021. [Online]. Available:
  \url{https://doi.org/10.1016/j.jde.2021.03.009}
\BIBentrySTDinterwordspacing

\bibitem{brockett1983asymptotic}
R.~W. Brockett, ``Asymptotic stability and feedback stabilization,'' in
  \emph{Differential geometric control theory ({H}oughton, {M}ich., 1982)},
  ser. Progr. Math.\hskip 1em plus 0.5em minus 0.4em\relax Birkh\"{a}user
  Boston, Boston, MA, 1983, vol.~27, pp. 181--191.

\bibitem{kvalheim2021necessary}
M.~D. Kvalheim and D.~E. Koditschek, ``Necessary conditions for feedback
  stabilization and safety,'' \emph{arXiv preprint arXiv:2106.00215}, 2021.

\end{thebibliography}

\end{document}